\documentclass[12pt,index]{amsart}

\usepackage{amssymb}
\usepackage{mathtools}

\usepackage{enumerate,xcolor,mathrsfs}
\usepackage[pagebackref,colorlinks,linkcolor=red,citecolor=blue,urlcolor=blue,hypertexnames=true]{hyperref}

\usepackage{ulem,xspace} 

\newcommand{\df}[1]{{\it{#1}}{\index{#1}}}

\renewcommand{\subset}{\subseteq}
\makeatletter
\newsavebox\myboxA
\newsavebox\myboxB
\newlength\mylenA

\newcommand*\xoverline[2][0.75]{%
    \sbox{\myboxA}{$\m@th#2$}%
    \setbox\myboxB\null
    \ht\myboxB=\ht\myboxA%
    \dp\myboxB=\dp\myboxA%
    \wd\myboxB=#1\wd\myboxA
    \sbox\myboxB{$\m@th\overline{\copy\myboxB}$}
    \setlength\mylenA{\the\wd\myboxA}
    \addtolength\mylenA{-\the\wd\myboxB}%
    \ifdim\wd\myboxB<\wd\myboxA%
       \rlap{\hskip 0.5\mylenA\usebox\myboxB}{\usebox\myboxA}%
    \else
        \hskip -0.5\mylenA\rlap{\usebox\myboxA}{\hskip 0.5\mylenA\usebox\myboxB}%
    \fi}
\makeatother

\newtheorem{theorem}            {Theorem}[section]
\newtheorem{corollary}          [theorem]{Corollary}
\newtheorem{proposition}        [theorem]{Proposition}

\newtheorem{lemma}              [theorem]{Lemma}
\newtheorem{remark}         [theorem]{Remark}

\newtheorem{assumption}     [theorem]{Assumption}

\newcommand{\C}{\mathbb{C}}

\newcommand{\cD}{\mathcal{D}}

\newcommand{\gv}{{\tt{g}}}
\newcommand{\bcdot}{{\bf{\cdot}}}
\newcommand{\tX}{\mathfrak{X}}

\newcommand{\eit}{e^{i\theta}}
\newcommand{\eito}{e^{i\theta_1}}
\newcommand{\emit}{e^{-i\theta}}

\newcommand{\cM}{\mathcal{M}}
\newcommand{\cN}{\mathcal{N}}
\newcommand{\cH}{\mathcal{H}}

\newcommand{\vg}{{\tt{g}}}
\newcommand{\matn}{M_n(\C)}
\newcommand{\matg}{M(\C)^g}
\newcommand{\matng}{M_n(\C)^{\vg}}
\newcommand{\matdg}{M_d(\C)^{\vg}}
\newcommand{\mattwo}{M(\C)^2}
\newcommand{\matntwo}{M_n(\C)^2}

\newcommand{\fP}{\mathfrak{P}}

\newcommand{\Langle}{\mathop{<}\!}
\newcommand{\Rangle}{\!\mathop{>}}

\newcommand{\xx}{\!\Langle x\Rangle}

\makeatletter
\def\moverlay{\mathpalette\mov@rlay}
\def\mov@rlay#1#2{\leavevmode\vtop{
		\baselineskip\z@skip \lineskiplimit-\maxdimen
		\ialign{\hfil$#1##$\hfil\cr#2\crcr}}}
\makeatother

\newcommand{\eo}{e_1}
\newcommand{\fo}{f_1}
\newcommand{\et}{e_2}
\newcommand{\ft}{f_2}

\newcommand{\Gp}{G}
\newcommand{\Ip}{J^{+}}

\newcommand{\Em}{E^{-}}
\newcommand{\Ep}{E^{+}}
\newcommand{\Fp}{F^{+}}
\newcommand{\Fm}{F^{-}}

\newcommand{\CC}{\mathbb{C}}

\newcommand{\EE}{F}

\newcommand{\dd}{\delta}
\newcommand{\altL}{\mathfrak{L}}
\newcommand{\mm}{{\tt{e}}}
\newcommand{\jj}{k}

\makeindex

\title[Reinhardt Spectrahedra]{Reinhardt Free Spectrahedra}

\author[McCullough]{Scott McCullough${}^1$}
\address{Department of Mathematics\\
  University of Flordia \\ Gainesville, FL}
\email{sam@ufl.edu}
\thanks{${}^1$Research Supported by NSF grant DMS-1764231.}

\author[Tuovila]{Nicole Tuovila} 
\address{Department of Mathematics\\
  University of Flordia \\ Gainesville, FL}
\email{n2vila@ufl.edu}


 \numberwithin{equation}{section}

\begin{document}

\begin{abstract}
 Free spectrahedra are natural objects in the theories
 of operator systems and spaces and completely positive
 maps. They also appear in various engineering applications.
 In this paper, free spectrahedra
 satisfying a Reinhardt symmetry condition are characterized
 graph theoretically. It is also shown that, for a simple
 class of such spectrahedra, automorphisms are trivial.
\end{abstract}

\maketitle

\thispagestyle{empty}

\section{Introduction}
 The articles \cite{longmaps,bestmaps}
 characterize, respectively,  bianalytic maps between free spectrahedra under hypotheses
 that are, in the sense of algebraic geometry, generic; and bianalytic
 maps between spectraballs. Free spectrahedra satisfying
 a Reinhardt condition represent 
 arguably the simplest
 class of free spectrahedra  not covered by these results.

 This paper contains two main results. Theorem~\ref{t:graph}
 gives a graph theoretic characterization of Reinhardt free
  spectrahedra. Theorem~\ref{t:mainR} shows that, for 
 a class of Reinhardt spectrahedra, automorphisms are trivial.

\subsection{Free spectrahedra}
 Let $\matn$ \index{$\matn$}  denote the $n\times n$ matrices with entries from $\mathbb{C}.$
 For $T\in \matn,$
the notation $T\succ 0$ (resp. $T\succeq 0$) indicates that $T$ is \df{positive definite}
 (resp. \df{positive semidefinite}).  \index{$\succ 0$} \index{$\succeq 0$}

 Fix a positive integer $\vg.$ Let $\matng$ \index{$\matng$} 
 denote the set of $\vg$-tuples of $n\times n$ matrices
 and let $x=(x_1,\dots,x_{\vg})$ denote  a tuple of indeterminants. Given $A=(A_1,\dots,A_{\vg})\in \matdg,$
 the expression \index{$L_A$}  \index{$\Lambda_A$}
\begin{equation*}
 L_A(x)  = I_d - \sum_{j=1}^{\vg}  A_j x_j -\sum_{j=1}^{\vg}  A_j^* x_j^* =
  I-\Lambda_A(x)-\Lambda_A(x)^*,
\end{equation*}
 where  $\Lambda_A(x)$ is the linear matrix polynomial,
\[
 \Lambda_A(x) = \sum_{j=1}^{\vg}  A_j x_j,
\]
 is a \df{monic linear pencil}. The formal expression 
 $L_A(x)\succeq 0$ is called a  \df{Linear Matrix Inequality, LMI.}

The pencil $L_A$  \df{evaluates} at a $\vg$-tuple $X=(X_1,\dots,X_{\vg}) \in \matng,$
 outputting an element of $\mathbb{S}_{dn}(\C),$ \index{$\mathbb{S}_{dn}(\C)$}
 the set of 
 self-adjoint matrix of size $dn,$  using the tensor (Kronecker)
  product as
\[
 L_A(X) = I_d\otimes I_n -\sum_{j=1}^{\vg} A_j\otimes X_j
 - \sum_{j=1}^{\vg} A_j^*\otimes X_j^*  = I_{dn} -\Lambda_A(X)
  -\Lambda_A(X)^*.
\]
 Alternately, $L_A(x)$ can be viewed
 as a matrix of  affine linear scalar polynomials and
 the evaluation $L_A(X)$ as entrywise affine combinations of the $X_j.$
 In the case $X=\zeta=(\zeta_1,\dots,\zeta_{\vg})\in \CC^{\vg} = M_1(\C)^{\vg},$
 the evaluation is simply  $L_A(\zeta) = I_d - \sum A_j \zeta_j 
- \sum A_j^*  \zeta_j^* \in \mathbb{S}_d(\C).$

  The solution set of the LMI $L_A(x)\succeq 0$
  at \df{level $n$} is the set \index{$\cD_A[n]$}
\[
 \cD_A[n] = \{X\in \matng: L_A(X)\succeq 0\}.
\]
  The set $\cD_A[1] = \{\zeta\in \CC^{\vg}: L_A(\zeta) \succeq 0\}$
 is known as a \df{spectrahedron}. Spectrahedra figure prominently
 in optimization and semidefinite programming (SDP) and real algebraic geometry.
 Semidefinite programs generalize linear programs and they 
 are now found in a diverse range of applications.  Control theory,
 combinatorics, and quantum information represent just a few
 \cite{BPT,Las,SIG,PNA,Wat,WSV}.

 The sequence
 $\cD_A =(\cD_A[n])_n$ is a \df{free spectrahedron}.
 Free spectrahedra appear in the theories of operator
 systems and spaces and completely positive maps.
 As matricial solution sets of LMIs,  
 they also arise in engineering contexts 
 \cite{SIG,convert-to-matin,emerge}.  In particular, free 
 spectrahedra typically produce tractable relaxations of
 problems involving spectrahedra \cite{BT-N,relax,DDSS,PSS}.
 
 Let \df{$\matg$} denote the sequence $(\matng)_n.$ Let \index{$\fP_A$}
 $\fP_A$ denote the (levelwise) interior of  the 
 free spectrahedron $\cD_A,$
\[
 \fP_A =\{X\in \matg:  L_A(X)\succ 0\}.
\]
  Both $\cD_A$ and $\fP_A$ are \df{free sets} in the following sense.
  Letting $S=(S[n])$ denote either $\cD_A$ or $\fP_A,$ we have
 $S\subseteq \matg$ in the sense that $S[n]\subseteq\matng$
 for each $n$ and 
\begin{enumerate}[(i)]
 \item if $X\in S[n]$ and $Y\in S[m],$ then \index{$X\oplus Y$} \index{$U^*XU$}
\[
 X\oplus Y : = (X_1\oplus Y_1, \cdots, X_{\vg}\oplus Y_{\vg}) \in S[n+m],
\]
 where
\[
 X_j\oplus Y_j :=\begin{pmatrix} X_j &0\\0 & Y_j \end{pmatrix} \in M_{n+m}(\C);
\]
\item if $X\in S[n]$ and $U$ is an $n\times n$ unitary matrix, then
\[
 U^* XU := \left ( U^*X_1 U,\dots, U^* X_{\vg}U \right )\in S[n].
\]
\end{enumerate}

\subsection{Free maps}
 A \df{free map} $f:\fP_A\to \fP_B$ between the interiors of
 free spectrahedra is a sequence of mappings $f=(f[n]),$ where \index{$f[n]$}
 $f[n]:\fP_A[n]\to \fP_B[n],$ such that 
\begin{enumerate}[(i)]
 \item if $X\in \fP_A[n]$ and $Y\in \fP_A[m],$ then
\[
 f[n+m](X\oplus Y) = f[n](X)\oplus f[m](Y); 
\]
\item if $X\in \fP_A[n]$ and $U$ is an $n\times n$ unitary matrix, then
\[
 f[n](U^* XU)=  U^*f[n](X) U. 
\]
\end{enumerate}
 We typically write $f$ in place of $f[n].$  The free map 
 $f$ is \df{analytic} if each $f[n]$ is analytic; and $f$ is
  \df{free bianalytic} or \df{free biholomorphic} if $f$ is 
 analytic and has a free analytic inverse.
 Motivations for studying free biholomorphic maps between free spectrahedra
 stem from  connections with matrix inequalities that arise 
 in systems engineering \cite{emerge,convert-to-matin}  
 and by analogy with rigidity theory in several complex variables. 
 Free biholomorphisms can also be viewed as nonlinear
 completely isometric maps.  We will often just say
 $f$ is bianalytic, dropping the adjective free.

\subsection{Overview and prior results}
 Given a tuple $E=(E_1,\dots,E_{\vg})$ of $k\times \ell$ 
 (not necessarily square) matrices, the spectrahedron 
 $\cD_A,$ where $A\in M_{k+\ell}(\C)^{\vg}$ is given by
\[
 A_j =\begin{pmatrix} 0 & E_j \\0& 0\end{pmatrix},
\]
 is a \df{spectraball}. It is routine to see that a tuple
 $X\in \matng$ is in $\cD_A[n]$ (resp. $\fP_A$) 
 if and only if  $\|\Lambda_E(X)\|\le 1$ (resp. $\|\Lambda_E(X)\|<1$).

 The results of \cite{longmaps,bestmaps} characterize
 bianalytic maps between $\fP_A$ and $\fP_B$
 under certain generic irreducibility hypotheses on 
 the tuples $A$ and $B.$  They also characterize bianalytic
 maps between spectraballs absent any additional hypotheses. 
 A canonical class of free spectrahedra not covered by these
 results are those that are circularly symmetric as defined
 in Subsection~\ref{sec:Reinhardt} immediately below. 
 Theorem~\ref{t:mainR} and its corollary, Corollary~\ref{c:mainR},
 provide evidence for the belief that a bianalytic map between circularly symmetric
 free spectrahedra have a simple algebraic form.

\subsection{Reinhardt spectrahedra}
\label{sec:Reinhardt}
 Let \df{$\mathbb T$} denote the unit circle in the complex plane.
 Thus $\mathbb T^{\vg}$ is the $\vg$-fold torus. 
 A free spectrahedron $\cD$ is {\it Reinhardt} 
 \index{Reinhardt spectrahedron}  if, for each $n,$
 each  $X=(X_1,\dots,X_\gv)\in \cD[n]$  and
 each  $\gamma =(\gamma_1,\dots,\gamma_{\gv})
 \in \mathbb T^{\gv},$ we have \index{$\gamma\bcdot X$}
\[
  \gamma \bcdot  X = \left ( \gamma_1 X_1,\dots, \gamma_{\gv} X_{\gv}
 \right ) \in \cD.
\] 
 A related, more general,  class of free spectrahedra is the circular
 ones. A free spectrahedron $\cD$ is \df{circular} if, for
 each $n$, each $X\in \cD[n]$ and each $\gamma\in \mathbb T$,
 we have $\gamma  X \in \cD[n].$

 This paper contains two main results. One is
 a graph theoretic characterization of Reinhardt free spectrahedra.
  See Theorem \ref{t:graph} in Section~\ref{s:FRgraph}.
   Theorem~\ref{t:mainR}   below is the other.

 Given positive integers $k,m,n$ and nonzero matrices $C_1\in M_{k,m}(\C)$
 and $C_2\in M_{m,n}(\C),$
 let $d=k+m+n$ and let $\EE_1$ and $\EE_2$ denote the block \df{$\EE$}
 $d\times d$ matrices 
\[
 \EE_1 =\begin{pmatrix}0&C_1&0\\0&0&0\\0&0&0\end{pmatrix}, \ \ \
 \EE_2=\begin{pmatrix}0&0&0\\0&0&C_2\\0&0&0\end{pmatrix}.
\] 
Scaling the $C_j$ corresponds to a simple scaling of $\fP_{\EE}.$ 
Thus there is no harm in assuming, as we now do, that 
 $\|C_j\|=1$ for $j=1,2.$  In particular,
\[
 C_1^* C_1 + C_2C_2^* \not \prec I
\]

The \df{free bidisk} is the free set $\{(X_1,X_2)\in \mattwo: \|X_j\|<1\}.$
 Note that the free bidisk is the interior of the  direct sum of the 
  spectraball $\{X\in M(\C)^1: \|X\|\le 1\}$ with itself.

\begin{theorem}\label{t:mainR}
 If $C_1^*C_1 + C_2C_2^*\preceq I,$ then $\fP_{\EE}$ is the free bidisk.
  Otherwise, 
 if $\varphi:\fP_{\EE} \to\fP_{\EE}$ is bianalytic, then there exists
 $\gamma =(\gamma_1,\gamma_2)\in \mathbb T^2$ such that 
 \[
 \varphi(x)=\gamma\cdot x = (\gamma_1 x_1,\gamma_2 x_2).
\]
\end{theorem}

 The proof of Theorem~\ref{t:mainR} 
 occupies Section~\ref{s:mainR}.
  Combining Theorem \ref{t:mainR} with results from \cite{proper}
 and \cite{billvolume} produces the following corollary.

\begin{corollary} \label{c:mainR}
 Suppose $\cD_B$ is a circular free spectrahedron and $\fP_{\EE}$ is not the free bidisk. 
  If $f:\fP_{\EE}\to \fP_B$
 is a free biholomorphism, then $f$ is linear.
\end{corollary}

\begin{proof}
 By a main result from \cite{billvolume}, if two circular free
 spectrahedra are biholomorphic, then they are linearly biholomorphic.
 Hence there is a biholomorphism $g:\fP_{\EE}\to \fP_B$ with 
 $g(0)=0.$ It follows that $h:=g^{-1}\circ f :\fP_{\EE}\to \fP_{\EE}$
 is biholomorphic and hence, by Theorem~\ref{t:mainR}, $h(0)=0.$ It follows that 
 $f(0)=0.$  By \cite[Theorem~4.4]{proper},  $f$ is linear.
\end{proof}

\subsection{Reader's guide}
 The body of this paper has three sections.
 Section~\ref{sec:graphstuff} isolates the purely graph theoretic
 preliminaries used in  the proof of Theorem~\ref{t:graph}, 
 the graph theoretic characterization of Reinhardt free spectrahedra.
  Theorem~\ref{t:graph}
 is stated and proved in Section~\ref{s:FRgraph}. Section~\ref{s:mainR}
 is devoted to the proof of Theorem~\ref{t:mainR}.

\section{Some graph theory preliminaries}
\label{sec:graphstuff}
 Let {$G=(V,\Ep)$
  be a \df{directed acyclic  graph} (\df{dag}),
  with vertices $V$ and edges $E^+\subset V\times V.$ \index{$G=(V,\Ep)$}
  We do not allow
 self cycles; that is, if $(v,w)\in \Ep,$ then $v\ne w.$  
 Acyclicity implies, if $(v,w)\in \Ep,$ then $(w,v)\notin \Ep.$
 
 A set $P=\{\Ep_1,\dots,\Ep_m\},$ with $\Ep_j\subseteq \Ep,$
 is a \df{partition} of $\Ep$ provided $\Ep_j\cap \Ep_k =\varnothing$ 
  for $j\ne k;$
 and  $\cup_{j=1}^m \Ep_j = \Ep.$  

 Let $\Em=\{(k,j): (j,k)\in E\}$ \index{$\Em$} and let $E=\Ep \cup \Em.$ \index{$E$}
  A cycle $L$ in the graph  $(V,E)$ is \df{$P$-neutral},
 or neutral with respect to the partition $P$, if,
 for each $j,$ the number of edges of $L$ that lie in $\Ep_j$
 and $\Em_j$ are the same.  In particular, a neutral cycle
 has an even number of edges and is not a cycle
 in the directed graph. We will say $\gamma\in \mathbb{T}^m$
 is \df{independent} if $\rho\in \mathbb Z^m$ and 
  $\prod_{j=1}^m \gamma_j^{\rho_j} =1$
  implies $\rho =0.$ Observe that
 $\gamma=(e^{i\theta_1},\dots,e^{i\theta_m})$  is independent if and only if  $\{\theta_1,\dots,\theta_m\}\subset \mathbb R^m$
 is linearly independent over $\mathbb{Z}$ modulo $2\pi$ if
 and only if $\{\theta_1,\dots,\theta_m,2\pi\}$ is linearly 
 independent over $\mathbb Z.$ In particular, independent
 $\gamma\in \mathbb{T}^m$ exist.

\begin{lemma}
 \label{l:neutral-loop}  
  Suppose $G=(V,\Ep)$ is a dag, $P=\{\Ep_1,\dots,\Ep_m\}$
 is a partition of $\Ep$ and $L=(v_0,v_1,\dots,v_n, v_{n+1}=v_0)$ 
 is a cycle in $(V,E)$. The following are equivalent:
\begin{enumerate}[(i)]
 \item \label{i:neutral} $L$ is  $P$-neutral;
 \item \label{i:allgamma}  for each $\gamma \in \mathbb{T}^{m}$
\begin{equation}
\label{e:prod1}
 \sigma_0 \, \sigma_1\, \sigma_2 \, \cdots \, \sigma_n = 1,
\end{equation}
 where $\sigma_j = \gamma_{s_j}$ if 
 $(v_j,v_{j+1})\in \Ep_{s_j}$  and $\sigma_j = \gamma_{s_j}^*$
 if $(v_{j},v_{j+1})\in \Em_{s_j};$ 
 \item \label{i:somegamma} item~\eqref{i:allgamma} holds for some 
 independent $\gamma \in \mathbb{T}^{m}.$
\end{enumerate}
\end{lemma}

\begin{proof}
 First suppose $L$ is $P$-neutral and $\gamma\in \mathbb{T}^m.$
 For each $1\le k\le m$ the number of edges of $L$ 
 in $\Ep_k$ and $\Em_k$ are the same. Hence 
 $\prod \{ \sigma_j: (v_j,v_{j+1})\in \Ep_k\cup \Em_k\}=1$
 and thus equation~\eqref{e:prod1} holds. Hence
 item~\eqref{i:neutral} implies item~\eqref{i:allgamma}.
 
 To complete the proof if suffices to show item~\eqref{i:somegamma}
 implies item~\eqref{i:neutral}. Accordingly suppose 
 $\gamma\in\mathbb{T}^m$ is independent and equation~\eqref{e:prod1}
 holds. For $1\le j\le m,$ let  $\Gamma_k^\pm  =\{j: (v_j,v_{j+1})\in E_k^\pm\}.$
 Let $\mu_k = |\Gamma_k^+| - |\Gamma_k^-|$ and note that 
\[
1 =  \sigma_0 \, \sigma_1\, \sigma_2 \, \cdots \, \sigma_n 
  = \prod_{k=1}^m  \left ( \prod_{j\in\Gamma_k^+} \sigma_j \, \prod_{j\in \Gamma_k^-} \sigma_j\right )
 = \prod_{k=1}^m  \gamma_k^{\mu_k}.
\]
 By the independence assumption, $\mu_k=0$ and thus
 $L$ is $P$-neutral.
\end{proof}

 We define a \df{Reinhardt partition} $P$ of the 
 graph $G=(V,\Ep)$ to be a partition 
 such that all cycles in the graph $(V,E)$ are $P$-neutral.
 For a function $\delta:V\to \mathbb{T},$ let $\delta_v = \delta(v).$

\begin{proposition}
 \label{p:mainGraph}
 If $P=\{\Ep_1,\dots,\Ep_m\}$ 
 is a Reinhardt partition of the 
 dag $G=(V,\Ep)$, then, for 
 each $\gamma\in\mathbb{T}^m,$  
 there exists a $\delta:V\to \mathbb{T}$  such that,
 for each $1\le j\le m$ and each $(v,w)\in \Ep_j,$
\begin{equation}
 \label{e:deltaV}
 \delta_v = \gamma_j \delta_w.
\end{equation}
\end{proposition}

The proof of Proposition~\ref{p:mainGraph}
uses the following lemma.

\begin{lemma}
 \label{l:removeone}
  Suppose $G=(V,\Ep)$ is a dag, $P=\{\Ep_1,\dots,\Ep_m\}$
 is a Reinhardt partition of $G,$ and $u\in V.$
 The directed graph $(W,\Fp)$ 
 obtained by removing the vertex $u$ (and all edges
 containing $u$) from $G$ is acyclic and moreover
 the partition $Q= \{\Fp_1,\dots,\Fp_m\},$
 where $\Fp_j = \Ep_j\cap \Fp,$ is a Reinhardt partition.
\end{lemma}

\begin{proof}
 Any directed cycle in $(W,\Fp)$ induces a directed
 cycle in $(V,\Ep).$ Likewise, any  cycle
 in $(W,\Fp\cup\Fm)$  is a cycle in $(V,\Ep\cup \Em)=(V,E).$
\end{proof}

\begin{proof}[Proof of Proposition \ref{p:mainGraph}]
 We argue by induction on the number of vertices of $G=(V,\Ep).$
 In the case $|V|=1$ there is nothing to prove. Now suppose
 $G=(V,\Ep)$ is a dag,  $P=\{E_1^+,\dots,E_m^+\}$ is
 a Reinhardt partition, $\gamma\in\mathbb{T}^{m},$ 
 the cardinality of $V$ is at least two 
 and, for every dag  $(W,\Fp)$ with $|W|<|V|,$
 every $n,$  every Reinhardt partition $\{\Fp_1,\dots,\Fp_n\}$
 of $\Fp,$ and every $\Gamma\in \mathbb{T}^n,$ there
 exists $\delta:W\to \mathbb{T}$ such that, for
 each $1\le s\le n$ and each $(v,w)\in \Fp_s,$ 
\begin{equation}
 \label{e:deltas}
  \delta_v = \Gamma_s \delta_w.
\end{equation}

 Since $G$ is a dag, there is a terminal vertex $u;$  that
 is, there does not exist a vertex $w$ such that $(u,w)\in \Ep.$
 Let $(W,F)$ denote the graph obtained by removing
 the vertex $u$ from $(V,E).$  Let $\Fp=F\cap \Ep$ 
 and $\Fm=F\cap \Em.$ 

 We first consider the case that $(W,F)$ is not connected.
 Thus $(W,F)$ is a disjoint union of $(Y_1,H_1)$ and  $(Y_2,H_2)$
 and  $Y_1\ne \varnothing \ne Y_2.$ 
 Consider the graphs $(W_\jj,F_\jj)$ obtained by adding $u$ back 
 into each of $(Y_\jj,H_\jj).$ Thus $W_\jj=Y_\jj\cup\{u\}$ and
\[
 F_\jj=H_\jj \cup \{(v,u)\in W_\jj\times W_\jj: (v,u)\in E\} 
 = H_\jj \cup \{(v,u)\in W_\jj\times W_\jj: (v,u)\in \Ep\}.
\]

 Note that $|W_\jj|<|V|$
 and, by Lemma~\ref{l:removeone}, $(W_\jj,\Fp_\jj)$
 is a dag and the (induced)
 partitions $Q^\jj=(F_\jj \cap \Ep_1,F_\jj\cap \Ep_2, \dots,F_\jj\cap \Ep_m)$
 are Reinhardt.    Applying the induction
 hypothesis to each produces 
 $\delta^k:W_k\to\mathbb{T}$ such that,
 for $(v,w)\in F_k \cap \Ep_j,$
\[
  \delta^k_v = \gamma_j \delta^k_w.
\]
Further,  in each case 
 we choose, without loss of generality, $\delta^k_u=1.$ 
 In this case, $\delta:V\to \mathbb{T}$ defined by
 $\delta_v=\delta^k_v$ for $v\in W_k$  (is well defined and) 
 satisfies equation~\eqref{e:deltaV}. 

 From here on, it is assumed that $(W,F)$ is connected. 
 From Lemma~\ref{l:removeone}, $(W,\Fp)$ is a dag
 and the partition
 $Q=\{\Fp_1,\dots,\Fp_m\}$ of $\Fp$ induced from $P$ is Reinhardt.  
 From the induction hypothesis, there is a $\delta:W\to \mathbb{T}$
 satisfying equation~\eqref{e:deltas}.  
 To complete the
 proof, it suffices to identify a $\delta_u\in\mathbb{T}$
 such that the function $\rho:V\to \mathbb{T}$
 defined by $\rho_v=\delta_v$ for $v\in W$ and $\rho_u =\delta_u$
 satisfies %
 equation~\eqref{e:deltas}.
 Thus, it suffices to show, if
 $(x,u)\in \Ep_a$ and $(y,u)\in \Ep_b,$ then
 $\delta_x \gamma_a^* = \delta_y \gamma_b^*,$
 since then $\delta_u=\delta_x \gamma_a^*$ is 
 independent of $x\in U=\{w\in V: (w,u)\in \Ep\}.$

 Since $(W,F)$ is connected, 
 there is a cycle $(u=v_0,x=v_1,v_2,\dots,y=v_n,u=v_{n+1})$ in $(V,E)$
 with   $(u,v_1)\in \Em_a$ and $(v_n,u)\in \Ep_b.$ 
 By hypothesis, $L$ is $P$-neutral. By
 Lemma~\ref{l:neutral-loop}, 
\begin{equation}
 \label{e:sig0n}
 \sigma_0 \, \sigma_1\, \sigma_2 \, \cdots \, \sigma_n = 1,
\end{equation}
 where $\sigma_j=\gamma_{s_j}$ if $(v_{j_1},v_{j+1})\in \Ep_{s_j}$
 and $\sigma_j = \gamma_{s_j}^*$ if $(v_{j},v_{j+1})\in \Ep_{s_j}.$
 Observe that $\sigma_0 = \gamma_a^*$ %
 and  $\sigma_n =\gamma_b.$ For each $j,$ either
\begin{enumerate}[(i)]
 \item  there is an 
 $1\le s\le m$ such that  $(v_j,v_{j+1})\in \Ep_s,$ 
 in which case $\delta_{v_j} \delta_{v_{j+1}}^* =\gamma_s= \sigma_j;$ or
 \item there is an $1\le s\le m$ such that $(v_{j},v_{j+1})\in \Em_s,$
  in which case $(v_{j+1},v_{j})\in \Ep_s$ and hence
  $\delta_{v_{j+1}}\delta_{v_{j}}^* = \gamma_s =\sigma_j^*.$
\end{enumerate}
 Hence, for each $1\le j \le n$, 
\begin{equation}
 \label{e:nomatter}
 \delta_{v_j}\delta_{v_{j+1}}^* = \sigma_{j}.
\end{equation}
 Combining equations~\eqref{e:nomatter} and \eqref{e:sig0n} 
 and using $\sigma_0=\gamma_a^*$ and $\sigma_n=\gamma_b,$
 gives,
\[
 \gamma_a \gamma_b^* = \sigma_0^* \sigma_n^* 
 = \sigma_1\, \sigma_2 \, \cdots \, \sigma_{n-1} 
   = \delta_{v_1}\delta_{v_{n}}^*. 
\] 
Equivalently, $\delta_x \gamma_a^* = \delta_{v_1} \gamma_a^* 
= \delta_{v_{n}} \gamma_b^*=\delta_y \gamma_b^*$
 as claimed.   
\end{proof}

\section{The structure of free Reinhardt LMI domains}
\label{s:FRgraph}
A tuple $A \in M_d(\C)^g$ is a \df{minimal defining tuple} for a free
 spectrahedron $\cD$ if $\cD = \cD_A$ and
 $B \in M_{d^\prime}(\C)^g$ satisfies $\cD=\cD_B,$ then 
 $d^\prime \geq d.$

\begin{theorem}
\label{t:graph}
Suppose $A\in M_d(\C)^{\vg}$ is a  minimal defining tuple for $\cD_A$ and each $A_s\ne 0.$

 The free spectrahedron $\cD_A$ is
 Reinhardt %
 if and only if there is an
 orthogonal  decomposition  $\C^d=\oplus_{\ell=1}^\mm  \cH_\ell$ such that,
 writing 
\[
 A_s = \begin{pmatrix} A_s(j,k) \end{pmatrix}_{j,k=1}^\mm
\]
 with respect to this decomposition and 
 letting 
\begin{enumerate}[(i)] 
 \item  $V=\{1,2,\dots,\mm\};$
 \item  $\Ip_s=\{(j,k)\in V\times V : A_s(j,k)\ne 0\}$
   for $1\le s\le \vg;$
 \item $E^+=\cup \Ip_s,$
\end{enumerate}
 the graph  $G=(V,E^+)$ is a dag  for which
 $\{\Ip_1,\dots,\Ip_\vg\}$ is a Reinhardt partition. 
\end{theorem}

\begin{remark}\rm
\label{r:connected}
 If the tuple  $A\in M_d(\C)^{\gv}$ has no reducing subspaces,
 then the graph $G$ is connected. The converse does not hold.
\end{remark}

\begin{proof}[Proof of Theorem~\ref{t:graph}]
 First suppose a decomposition $\C^d = \oplus_{\ell=1}^\mm \cH_\ell$
  with the stated properties exists.
 Hence, 
 given $\gamma\in\mathbb T^{\vg},$ 
  there exists $\delta:V\to \mathbb{T}$ such that,
 for each $1\le s\le \mm$ and each  $(j,k)\in \Ip_s,$
\[
 \delta_j = \gamma_s \delta_k,
\]
 by Proposition~\ref{p:mainGraph}. Let $I_\ell$ denote the 
 identity on $\cH_\ell$ and $U$ the diagonal matrix
 $\oplus_{\ell=1}^\mm \delta_\ell I_\ell.$ Given $1\le s\le \vg,$
 and $1\le j,k\le \mm,$
\[
 (U A_s U^*)(j,k) =\delta_j A_s(j,k)\delta_k^* 
  =\begin{cases} \gamma_s A_s(j,k) & \ \ (j,k)\in \Ip_s \\
    0 & \ \ (j,k)\notin \Ip_s. \end{cases}
\]
 Hence $U A_s U^* = \gamma_s A_s$ and $U AU^* =\gamma\bcdot A.$
  Thus, for $X\in \matng,$
\[
 \Lambda_A(\gamma\bcdot X) =  \sum_{s=1}^{\vg} A_s \otimes \gamma_s X_s 
   = \sum_{s=1} \gamma_s A_s \otimes X_s = \Lambda_{\gamma \bcdot A}(X)
  = (U\otimes I) \Lambda_A(X) (U\otimes I)^*
\]
 and therefore $X\in \cD_A$ if and only if $\gamma\bcdot X\in \cD_A.$
 Hence $\cD_A$ is Reinhardt. 

  Now suppose $\cD_A$ is Reinhardt and  $A$ is minimal for $\cD_A.$
Choose $\gamma\in \mathbb{T}^{\vg}$ independent as defined
 immediately before Lemma~\ref{l:neutral-loop}.
 Observe that $\cD_A=\cD_{\gamma\cdot A}$ by 
 the Reinhardt assumption.  Thus, by the 
 Linear Gleichstellensatz \cite[Proposition~2.2]{circular}, there exists 
  a $d\times d$ unitary matrix $U$ such that
\[
 UAU^* =  \gamma\bcdot A.
\]
It can be assumed, without loss of generality, that 
\begin{equation}
\label{e:U}
U = \lambda_1 I_{n_1}\oplus \lambda_2 I_{n_2} \oplus \cdots \oplus
\lambda_{\mm} I_{n_\mm},
\end{equation}
 where the $\lambda_j$ are distinct unimodular
  numbers,  the $n_j$ are positive integers
 and  $I_{n_j}$ is the identity matrix on $\C^{n_j}.$
 With respect to the block decomposition of $U,$ write
\[
    A_s = (A_s(j,l))_{j,l=1}^\mm.
\]
 Thus, for each $1\le j,l \le \mm$ and each $1\le s\le \vg,$
\[
  \lambda_j A_s(j, l) \lambda_l^*  = \gamma_s A_s(j,l),
\]
 and in particular, if $A_s(j,l)\ne 0,$ then 
 $\lambda_j = \gamma_s\lambda_l.$ 

 Let $V=\{1,\dots,\mm\}$ and let
 $\Ip_s =\{(v,w)\in V\times V: A_s(v,w)\ne 0\}$ for $1\le s\le \vg.$
 Thus, if $(v,w)\in \Ip_s,$ then 
\begin{equation}
 \label{e:gamma-lambda}
\lambda_v = \gamma_s \lambda_w.
\end{equation}
 Let $\Gp$ denote the directed graph with vertices $V$
 and edges $\Ep =\cup_{s=1}^{\vg} \Ip_s.$  As usual, let $\Em=\{(w,v):(v,w)\in E\}$
 and $E=\Ep\cup\Em.$

We claim
\begin{enumerate}
  \item \label{i:c1} $\Ip_s\cap \Ip_t=\emptyset$ for $s\ne t.$
       Thus $J=\{\Ip_1,\dots,\Ip_\vg\}$ is a partition of $\Ep.$
  \item \label{i:c2} If $(v,w)\in \Ep,$ then $(w,v)\notin \Ep.$ Hence
    $\Ep\cap \Em = \varnothing.$
    In particular, $G=(V,\Ep)$ has no self cycles.
  \item  \label{i:c3}  If $L$ is a cycle in $(V,E),$  then $L$ is
    $J$-neutral. In particular, $\Gp=(V,\Ep)$ is a dag
    and $J$ is a Reinhardt partition of $\Gp.$
\end{enumerate}

 To prove item~\eqref{i:c1}, suppose $(v,w) \in \Ip_s \cap \Ip_t.$ 
 From equation~\eqref{e:gamma-lambda} 
\[
\lambda_w \gamma_s = \lambda_v = \lambda_w \gamma_t.
\]
Hence  $s = t.$ Thus, if $s\ne t,$ then 
 $\Ip_s\cap \Ip_t=\varnothing.$

 Note that item~\eqref{i:c2} is a special case of item~\eqref{i:c3},
 but it is convenient to prove item~\eqref{i:c2} separately to simplify the
 logic of the statement of item~\eqref{i:c3}.  If $v,w\in V$
 are distinct and $(v,w),(w,v)\in \Ep,$ then there exist 
 $1\le s,t\le \vg$ with $(v,w)\in \Ip_s$ and $(w,v)\in \Ip_t.$ It follows
 that $\lambda_v = \gamma_s \lambda_w$ and $\lambda_w=\gamma_t \lambda_v.$
 Hence, $\gamma_s \gamma_t =1,$ contradicting the choice
 of $\gamma.$

 Turning to the proof of item~\eqref{i:c3},
 suppose  $L=(v_0,v_1, v_2, \dots, v_n,v_{n+1}=v_0)$ is a cycle
 in $(V,E).$ 
   Thus, there exists 
$s_0, \dots, s_n$ such that either 
 $(v_k,v_{k+1})\in \Ip_{s_k}$ or $(v_{k+1},v_k)\in\Ip_{s_k}$ 
 for $0\le k\le n.$  In the first case, let $\sigma_k=\gamma_{s_k}$
and in the second, let $\sigma_k = \gamma_{s_k}^*.$ In either case,
 equation (\ref{e:gamma-lambda}) implies
\[
 \lambda_{v_k} = \sigma_k  \lambda_{v_{k+1}},
\]
 for $0\le k\le n.$
Therefore, $\sigma_0 \sigma_1 \cdots \sigma_n = 1.$ Hence, 
by Lemma~\ref{l:neutral-loop}, 
  $L$ is $J$-neutral. It follows that $J$
 is a Reinhardt partition of $G.$  Since each cycle in the graph $(V,E)$ 
 is  $J$-neutral, it also follows that $G$ is acyclic as a directed graph,
 completing the proof of the theorem.
\end{proof}

\begin{remark}\rm
 The argument above leads to a conceptually simpler proof of
 \cite[Theorem~1.3]{circular}
 on circular spectrahedra that drops the irreducibility
 hypothesis. Namely, in that case
 one argues that the graph is a disjoint union of
 directed paths  by showing each vertex
 has at most one outgoing edge and one incoming edge.

 The statement of the result is: If $\cD_A$ is circular and 
 $A$ is minimal for $\cD_A,$ then, with respect
 to an appropriate orthogonal decomposition,
\begin{equation}
\label{e:formA}
  A_s =\begin{pmatrix} 0& A_s(1) &  0 &0 & \cdots &0 \\
    0&0& A_s(2)&0 & \cdots & 0\\
  \vdots & \vdots & \ddots & \ddots & \cdots & \vdots \\
 0&0&0&0 & \cdots & A_s(m-1) \\
0&0&0&0&\cdots&0 \end{pmatrix}.    
\end{equation}

\begin{proof}
 Just as in \cite{circular}, choose a $\theta$
 relatively prime to $2\pi,$ set $\gamma=\exp(i\theta)$ and 
 obtain a unitary $U$ such that $U A U^* = \gamma\, A.$
 Express this unitary $U$ as in equation~\eqref{e:U}.
 Express the $A_s$ accordingly.  In this case
 the analog of equation~\eqref{e:deltas}, is 
\[
 \lambda_v = \gamma \lambda_w,
\]
 if there is an $s$ such that $A_s(v,w)\ne 0.$ 
 Let $V=\{1,\dots,m\},$ let $\Ep$ denote the pairs $(v,w)$
 such that there is an $s$ such that $A_s(v,w)\ne 0$
 and consider the directed graph $G=(V,\Ep).$ 

  It is straightforward to prove that $G$ is a dag; that is,
 has no directed cycles. Indeed, if 
 $(\lambda_{v_0},\dots,\lambda_{v_{n+1}} =\lambda_{v_0})$
 is a (directed) cycle, then
\[
 \lambda_{v_0} = \gamma^n \lambda_{v_0}
\]
 and thus $\gamma^n=1,$ contradicting the choice of $\theta.$

   To show that $G$ is a disjoint union of  paths, 
 we first show each vertex has at most one incoming and
 one outgoing edge.  Indeed, if $(v,w),(v,u)\in \Ep,$ then, 
 by equation~\ref{e:gamma-lambda},
    \[
        \gamma \lambda_{w} = \lambda_v = \gamma \lambda_u.
    \]
    Thus, $w=u,$ since the $\lambda_j$ are distinct. Thus the
  vertex $v$ has at most one edge directed away from it.
  Similarly, if both $(w,v), (u,v) \in E,$ then, by equation~\ref{e:gamma-lambda},
    \[
        \lambda_{w} = \gamma \lambda_v = \lambda_u.
    \]
    Thus, $w=u$ and the vertex $v$ has at most one edge directed into it.

Now, let $P = (v_1, v_2, \dots, v_k)$ 
 be a maximal directed path  in $G.$  In
 particular, the $v_j$ are distinct and $(v_j,v_{j+1})\in \Ep$
 for $1\le j\le k-1.$   There are no additional edges
 between the vertices $\{v_1,\dots,v_k\}$ since
 already in the path $P$ the vertices $v_2,\dots,v_{k-1}$ 
 have degree two and if $(v_1,v_k)\in \Ep,$
 then $v_k$ has two edges directed into it; and if $(v_k,v_1)\in \Ep$
 then $G$ contains a directed cycle. Similarly, 
 if $u\in V\setminus\{v_1,\dots,v_k\},$ then neither
 $(u,v_1)$ nor $(v_k,u)$ are in $\Ep$
 by maximality of $P$; and neither $(v_1,u)$ nor $(u,v_k)$ are in $\Ep$
 since otherwise $v_1$ would have two edges directed away from
 it or $v_k$ would have two edges directed into it.
 Likewise for $2\le \ell \le k-1,$ neither $(u,v_\ell)$
 nor $(v_\ell,u)$ is in $\Ep$ as otherwise the total
 degree of the vertex $v_\ell$ would be at least three.   Hence
 $P$ determines a connected component of the 
 undirected graph $(V,E)$ and  is a maximal  directed
 path. 

 If $v$ is a vertex such that for all other vertices
 $u$ both  $(u,v)\notin \Ep$ and $(v,u)\notin  \Ep,$
 then $A_s(u,v)=0$ and $A_s(v,u)=0$ for all vertices $u.$
 It follows that $\cH_v\ne\{0\}$ reduces $A$ and moreover
 $A=0$ on $\cH_v,$ contradicting the minimality assumption on $A.$

 It now follows that each vertex appears in 
 exactly one  maximal directed
 path (of length at least two). Moreover, this path
 is a connected component of the undirected graph $(V,E).$
 Hence the undirected graph  $(V,E)$ is a union
 of directed paths, each of which is a connected component
 of $(V,E).$
  Finally, each connected component 
 component of the graph corresponds to a reducing subspace
 for $A.$ Indeed, for a component determined by 
 the directed path $P=(v_1,\dots,v_k),$ the subspace
 $\oplus_{j=1}^k \cH_{v_k}$ reduces $A$ and on 
 this subspace $A$ has the form of equation~\eqref{e:formA}.
\end{proof}
\end{remark}

\section{Proof of Theorem~\ref{t:mainR}}
\label{s:mainR}
 This section is devoted to the proof of Theorem~\ref{t:mainR}.
 Theorem~4.4 in  \cite{proper} says that says that
 a free bianalytic map between circular domains that
 sends $0$ to $0$ is linear. Correspondingly, much of the
  proof of Theorem~\ref{t:mainR} reduces to showing 
 the free bianalytic automorphism $\varphi$ satisfies
 $\varphi(0)=0.$  Evaluation of
 $\varphi$ at a nilpotent tuple involves only finitely
 many terms in its power series expansion.
The domain of $\varphi$
 naturally extends to include all nilpotent 
 tuples. 
 Since $\varphi$ is bianalytic, it
 maps nilpotent tuples in the boundary of $\fP_{\EE}$
 to tuples in the boundary of $\fP_{\EE}.$  See
 Subsection~\ref{ssec:power}. Subsection~\ref{ssec:boundaryfPE} provides a
 characterization of this boundary.
 The main result of Subsection~\ref{ssec:order2}, Lemma~\ref{l:eitheror},
 obtained by evaluating $\varphi$ on tuples in the boundary
 of $\fP_\EE$ on two nilpotent of order two, imposes
 serious restrictions on the relation between
 the coefficients in the affine linear portion (in the power
 series representation)  of $\varphi.$ These relations,
 combined with variations on the Caratheodory interpolation
 theorem developed in Subsection~\ref{ssec:Cint}, 
 are used in conjunction with evaluation of $\varphi$
 on tuples in the boundary of $\fP_\EE$ that
 are nilpotent of order three to complete the proof
 of Theorem~\ref{t:mainR} in  Subsection~\ref{ssec:order3}.

\subsection{Power series representations}
\label{ssec:power}
 Let $\xx$ \index{$\xx$} denote the 
  \df{words} in the freely noncommuting variables
 $x=\{x_1,x_2\}.$   
A \df{word}
  $w\in \xx$ is \df{evaluated} at a tuple $X\in\matng$ in the 
 evident way,
\begin{equation}
\label{e:word}
w=x_{j_1} x_{j_2} \cdots x_{j_m} \mapsto  w(X)=X^w = X_{j_1} X_{j_2} \cdots X_{j_m}.
\end{equation} \index{$w(X)$} \index{$X^w$}

 Given $f_w\in \CC$ for $w\in \xx,$ the (formal) expression, 
\[
 f(x) = \sum_{w\in \xx} f_w w,
\]
is a \df{power series}. Given a tuple $X\in \matntwo,$ 
\begin{equation*}
 f(X) = \sum_{m=0}^\infty  \left [\sum_{|w|=m} f_w X^w \right ],
\end{equation*}
 provided the sum (as indicated) converges.  Here
 $|w|$ denotes the \df{length} of the word $w.$  
 For instance the word of equation~\eqref{e:word}
 has length $m.$
 
 Given $\epsilon>0,$ let $\cN_\epsilon = (\cN_\epsilon[n])_n$
 denote the sequence of sets with
\[
 \cN_\epsilon[n] =\{X\in \matntwo: \|X_j\|<\epsilon\}.
\]
 The free set $\cN_\epsilon$ is a \df{neighborhood of zero}.
 Since $\fP_{\EE}$ contains a neighborhood of $0,$
 by results in \cite{KVV,billvolume}, 
 if $f:\fP_{\EE}\to M(\C)$ is a free function that is \df{bounded}
 (meaning there is an $R$ such that $\|f(X)\|<R$ for all $X\in \fP_{\EE}$),
 then $f$ admits a power series expansion that converges on
 some neighborhood of $0.$

 A tuple $T\in \mattwo$ is \df{nilpotent of order at most $N$} if $T^w=0$ 
 for each $w\in \xx$ with $|w|=N.$ 

\begin{lemma}\label{l:nilp-poly}
 If $f:\fP_{\EE}\to M(\C)$ is a bounded 
 free analytic function and $X\in \fP_{\EE}[n]$ is nilpotent
 of order at most  $N,$ 
 then $f_X:\mathbb D \to M_n(\C)$  defined by $f_X(z)= f(zX)$ is a polynomial
 of degree at most $N-1.$  Thus $f_X$ extends canonically
 to all of $\CC$ and hence the domain of $f$ naturally
 extends to contain all nilpotent tuples.
\end{lemma}

\begin{proof}
 For $z\in \CC$ with $|z|$ sufficiently small, it is evident that
  $f_X(z)=f(zX)$  is a polynomial (of degree at most $N-1$) by considering the convergent
 power series expansion for $f.$ By analyticity, $f_X(z)=f(zX)$
 for all $z\in \mathbb D.$ If  $X$ is an arbitrary nilpotent tuple (not necessarily in $\fP_{\EE}$), then
  there is a $\delta>0$ such that $\delta X\in \fP_{\EE}$
  and we define  $\varphi(X) = f_{\delta X}(\frac{1}{\delta}).$
\end{proof}

\subsection{The boundary of $\fP_{\EE}$}
\label{ssec:boundaryfPE}
 Let \df{$\varphi_j$} denote the coordinate functions of $\varphi.$ 
 Thus each $\varphi_j: \fP_{\EE}\to M(\C)$ is a bounded 
 free analytic mapping to which Lemma~\ref{l:nilp-poly}
 applies.

\begin{lemma}
 \label{l:deltodel}
 If $X$ is nilpotent and in the boundary of $\fP_{\EE},$  then 
 $\varphi(X)$ is in the boundary $\fP_{\EE}.$
\end{lemma}

\begin{proof}
 There is an $n$ such that $X$ is in the boundary of $\fP_{\EE}[n].$
 For $0<t<1,$ the tuple $tX$ is in
 $\fP_{\EE}[n]$ and hence  $Y:=\varphi(X)$ is in the closure
 of $\fP_{\EE}[n].$   On the other hand, the inverse $\psi$
 of $\varphi$ maps  $\fP_{\EE}$ into itself.
 Thus, if $Y\in \fP_{\EE}[n],$ then $\psi(Y)=X\in \fP_{\EE}[n].$ This
 contradiction shows $Y$ is in the boundary of $\fP_{\EE}[n].$
\end{proof}

\begin{lemma}
\label{l:inDA-alt}
 A tuple $(X_1,X_2)\in M_n(\C)^2$ is in  $\fP_{\EE}[n]$ if and only if
\[
 X_1^*X_1\otimes C_1^*C_1 + X_2X_2^* \otimes C_2C_2^* \prec I,
\]
 and is in the boundary of $\fP_{\EE}[n]$ if and only if 
\[
 X_1^*X_1\otimes C_1^*C_1 + X_2X_2^* \otimes C_2C_2^* \preceq I,
\]
 and there is a non-zero vector $\gamma\in \C^n\otimes \C^d $ such that
\[
 \left [X_1^*X_1\otimes C_1^*C_1 + X_2X_2^* \otimes C_2C_2^* \right]\, \gamma
 =\gamma.
\]
\end{lemma}

\begin{proof}
 A standard Schur complement argument proves the lemma.
\end{proof}

\begin{lemma}
\label{l:freethebidisk}
 The set $\fP_{\EE}$ is a subset of the free bidisk.  
 Moreover,
  if  $C_1^*C_1 +C_2C_2^*\preceq I,$ then $\fP_{\EE}$ is the free bidisk.
\end{lemma}

\begin{proof}
 Since $\|C_j\|=1$ for $j=1,2,$ it is immediate
 from Lemma~\ref{l:inDA-alt}  that 
 $\fP_{\EE}$ is contained in the free bidisk.  

 If also $C_1^*C_1 + C_2 C_2^*\preceq I,$ and $X$ is in the free bidisk,
 then 
\[
 X_1^*X_1 \otimes C_1^*C_1 + X_2X_2^*\otimes C_2C_2^*
\prec I \otimes C_1^*C_1 + I \otimes C_2C_2^* \preceq I.
\]
Thus, by Lemma~\ref{l:inDA-alt}, $X\in \fP_{\EE}.$
\end{proof}

\begin{assumption}
\label{a:notleI}
In view of Lemma~\ref{l:freethebidisk}, from here on it is assumed that $C_1^*C_1 + C_2C_2^* \not\preceq I.$
\end{assumption}

\subsection{Order two nilpotency}
\label{ssec:order2}
 In this section some information obtainable by evaluating
 $\varphi$ on order two nilpotent
  tuples  in the boundary of $\fP_{\EE}$ is collected.

 Let,  for $t=(t_1,t_2)\in\C^2,$  \index{$\tX(t)$}
\[
 \tX(t) := \left ( \begin{pmatrix} 0 & t_1\\ 0 &0\end{pmatrix},
 \,  \begin{pmatrix} 0 & t_2 \\0&0\end{pmatrix} \right ).
\]

\begin{lemma}
\label{l:tXt-in-boundary}
 The tuple $\tX(t)$ lies in $\fP_{\EE}[2]$ if and only if $|t_j|< 1$ 
 for each $j$ and it  is in the boundary
 of $\fP_{\EE}[2]$ if and only if $|t_j|\le 1$ and either $|t_1|=1$
 or $|t_2|=1.$
\end{lemma}

\begin{proof}
 Observe 
\[
 C_1^*C_1\otimes \tX_1(t)^*\tX_1(t)
  + C_2 C_2^* \otimes \tX_2(t)\tX_2(t)^*
 = \begin{pmatrix} |t_2|^2 C_2 C_2^* & 0 \\ 0 & |t_1|^2 C_1^*C_1 
 \end{pmatrix}
\]
 and apply  Lemma~\ref{l:inDA-alt}.
\end{proof}

 Recall, since the domain of $\varphi$ contains a neighborhood of $0$
 and its codomain is bounded, $\varphi_j$ has a power series expansion,
\begin{equation}
\label{e:vphij}
 \varphi_j(x) = b_j +\sum_{k=1}^2 \ell_{j,k} x_k  
   + \sum_{|w|\ge 2} (\varphi_j)_w w
\end{equation}
 and is defined for nilpotent tuples, whether in $\fP_{\EE}$ or not. 
 In view of equation~\eqref{e:vphij},  
\begin{equation}
\label{e:vphijtX}
\varphi_j(\tX(t)) = \begin{pmatrix} b_j & \sum \ell_{j,k}t_k 
  \\ 0 & b_j \end{pmatrix}.
\end{equation}
 Let \index{$L$}
\begin{equation}
 \label{d:L}
 \altL= \begin{pmatrix} \ell_{1,1}&\ell_{1,2}\\ \ell_{2,1} & \ell_{2,2}
 \end{pmatrix}.
\end{equation}
Since $\varphi$ is an automorphism, the matrix $\altL$ is invertible.

\begin{lemma}
\label{l:eitheror}
  Either 
 \begin{enumerate}[(a)]
  \item \label{i:a}
    $\ell_{1,2}=0=\ell_{2,1}$ and $\ell_{1,1}\ne 0 \ne \ell_{2,2}$; or
  \item \label{i:b}
      $\ell_{1,2}\ne 0 \ne \ell_{2,1}$ and $\ell_{1,1} = 0 = \ell_{2,2}.$
 \end{enumerate}
  
 Moreover, in the first case
\begin{equation}
 \label{e:elljj}
 \|\begin{pmatrix} b_j & \ell_{j,j} \\ 0 & b_j \end{pmatrix} \| =1
\end{equation}
 and in the second case,
\begin{equation}
 \label{e:elljk}
 \| \begin{pmatrix} b_j & \ell_{j,k} \\ 0 & b_j \end{pmatrix}\|=1,
\end{equation}
 for $k\ne j.$
\end{lemma}

\begin{proof}
Let   
\[
  F(t)  =  \varphi_1(\tX(t))^* \varphi_1(\tX(t))\otimes C_1^*C_1 
   + \varphi_2(\tX(t))\varphi_2(\tX(t))^*    \otimes C_2 C_2^*.
\]
From equation~\eqref{e:vphijtX}, 
\[
\begin{split}
  F(t) %
  = & \begin{pmatrix} |b_1|^2  &  b_1^*(t_1\ell_{1,1}+t_2\ell_{1,2})\\
   b_1(t_1\ell_{1,1}+t_2\ell_{1,2})^*  & |b_1|^2 + |t_1\ell_{1,1}+t_2\ell_{1,2}|^2 
 \end{pmatrix} \otimes C_1^*C_1 \\
  & +  \begin{pmatrix} |b_2|^2 + |t_1 \ell_{2,1} + t_2\ell_{2,2}|^2 
   & b_2^* (t_1 \ell_{2,1} + t_2\ell_{2,2}) \\ b_2 (t_1 \ell_{2,1} + t_2\ell_{2,2})^*
    & |b_2|^2 \end{pmatrix} \otimes C_2 C_2^*.
\end{split}
\]

  By Lemma~\ref{l:inDA-alt},  $F(t)$ has norm at most one for all $-1\le t_j\le 1.$
 It follows that
\begin{equation}
\label{e:G}
 G= \frac{F(1,1)+F(1,-1)}{2}
\end{equation}
 has norm at most one. 
 On  the other hand,
\begin{equation}
\label{e:G-alt}
 G =F(1,0) + \begin{pmatrix} 0&0\\0&  |\ell_{1,2}|^2 \end{pmatrix} 
   \otimes C_1^*C_1 +   \begin{pmatrix}  |\ell_{2,2}|^2 & 0 \\0 &0 \end{pmatrix}
  \otimes C_2 C_2^*.
\end{equation}
 Since all the terms on the right hand side of
 equation~\eqref{e:G-alt} are positive semidefinite
 and since  $F(1,0)$ has norm one 
   (by Lemma~\ref{l:tXt-in-boundary}),  it follows that
 $G$ has norm at least one. Thus $G$ has norm one.

 Since $F(1,0)$ has norm one,   there is a vector $\gamma\in \C^2\otimes\C^d$ 
 such that $F(1,0)\gamma =\gamma.$ Write 
\[
 \gamma =\sum_{j=1}^2 e_j \otimes \gamma_j.
\]
 From positivity considerations,
\begin{equation}
\label{e:somethingsare0}
0 = \left ( \begin{pmatrix} 0&0\\0& |\ell_{1,2}|^2 \end{pmatrix} 
   \otimes C_1^*C_1 + \begin{pmatrix}  |\ell_{2,2}|^2 & 0 \\0 &0 \end{pmatrix}
 \otimes C_2C_2^* 
 \right ) \gamma
 = \begin{pmatrix}  |\ell_{2,2}|^2 C_2C_2^* \gamma_1 \\   |\ell_{1,2}|^2 
   C_1^*C_1\, \gamma_2 \end{pmatrix}.
\end{equation}
  Hence either $\ell_{2,2}=0$ or $C_2C_2^* \gamma_1=0$  and either
 $\ell_{1,2}=0$ or $C_1^*C_1\gamma_2=0.$ It is not the case that 
 $\ell_{1,2}=0=\ell_{2,2},$   since $\altL$ is invertible. 
 From equation \eqref{e:G}, 
\[
 1 = \langle G\gamma,\gamma\rangle = 
  \frac12 \left ( \langle F(1,1)\gamma,\gamma\rangle + \langle F(1,-1)\gamma,
  \gamma\rangle \right ) \\
\le  1.
\]
Hence $F(1,\pm 1)\gamma=\gamma$ too.  It follows that
\[
0=  (F(1,1)-F(1,0))\gamma  =
 \begin{pmatrix} \left [|\ell_{2,1}+\ell_{2,2}|^2 - |\ell_{2,1}|^2\right ] 
   C_2C_2^*\gamma_1 +  \left [b_1^* \ell_{1,2} C_1^*C_1  +   b_2^* \ell_{2,2} C_2C_2^* \right ] \gamma_2 \\
  \left[|\ell_{1,1}+\ell_{1,2}|^2-|\ell_{1,1}|^2\right ] C_1C_1^*\gamma_2
  + \left [b_1\ell_{1,2}^* C_1^*C_1 + b_2 \ell_{2,2}^* C_2C_2^* \right] \gamma_1
     \end{pmatrix}.
\]
In view of equation~\eqref{e:somethingsare0}, 
\begin{equation}
\label{e:somemorethings0}
0 = \left ( F(1,1)-F(1,0)\right ) \gamma
 = \begin{pmatrix}    b_2^* \ell_{2,2} C_2C_2^*  \gamma_2 \\
  b_1\ell_{1,2}^* C_1^*C_1  \gamma_1
     \end{pmatrix}.
\end{equation}
 If $\ell_{1,2}\ne 0 \ne\ell_{2,2},$ then,
 from equations~\eqref{e:somethingsare0} and 
 \eqref{e:somemorethings0}, 
 $C_1^*C_1 \gamma_2= 0=b_1C_1^*C_1 \gamma_1$
 and $C_2C_2^*\gamma_1=0=b_2 C_2C_2^*\gamma_2.$
From $\gamma = F(1,0)\gamma,$ we thus obtain
 the contradiction
\[
  \gamma_1 = \left [ |b_1|^2 C_1^* C_1 + \left (|b_2|^2 + |\ell_{2,1}|^2\right)
  \otimes C_2C_2^* \right ] \, \gamma_1 \, + \, 
  \left[ b_1^* \ell_{1,1} C_1^*C_1 + b_2^* \ell_{2,1} C_2 C_2^* \right ] \, \gamma_2
  =0 
\]
and
\[
\gamma_2 = \left [ \left(|b_1|^2 +|\ell_{1,1}|^2\right ) C_1^*C_1 \, +
   \, |b_2|^2 C_2 C_2^* \right ] \, \gamma_2 \, + \, 
    \left [  b_1\ell_{1,1}^* C_1^*C_1  + b_2\ell_{2,1}^* C_2C_2^* \right ]\,
     \gamma_1 =0.
\]
 Thus  exactly one of $\ell_{1,2}, \, \ell_{2,2}$ is $0.$

 If $\ell_{1,2}=0,$ then $\ell_{2,2}\ne 0$ and both $C_2^*\gamma_1=0$
 and $b_2C_2^* \gamma_2=0$ by equations~\eqref{e:somethingsare0} and 
 \eqref{e:somemorethings0}.
 It follows that
\[
 \varphi_2(\tX(1,0)) \varphi_2(\tX(1,0))^* \otimes C_2 C_2^* \gamma =0
\]
and therefore
\[
\gamma= F(1,0)\gamma=  \varphi_1(\tX(1,0))^* \varphi_1(\tX(1,0)) \otimes C_1^*C_1 \gamma.
\]
 Using the fact that $\| U\otimes V\|=\|U\|\, \|V\|$ for matrices
 $U$ and $V$ (where the norms are the operator norms), 
 $\|\varphi_1(\tX(1,0))\|\ge 1.$ On the other $F(1,0)\preceq I$
 implies $\|\varphi_1(\tX(1,0))\| \le 1.$  Thus $\|\varphi_1(\tX(1,0))\|=1.$
 Note that, in this case,
\[
 \varphi_1(\tX(1,0)) =\begin{pmatrix} b_1 & \ell_{1,1} \\ 0 & b_1\end{pmatrix}.
\]

 If instead $\ell_{2,2}=0$ and $\ell_{1,2}\ne 0,$ then $C_1 \gamma_2=0$
 and $b_1 C_1^*C_1\gamma_1=0.$ It follows that 
\[
 \varphi_1(\tX(1,0))^* \varphi_1(\tX(1,0)) \otimes C_1^*C_1 \gamma =0
\]
and hence 
\[
 \gamma = \varphi_2(\tX(1,0))\varphi_2(\tX(1,0))\otimes C_2 C_2^* \gamma.
\]
Thus 
\[
 \varphi_2(\tX(1,0)) = \begin{pmatrix} b_2 & \ell_{2,1} \\ 0 & b_2 \end{pmatrix}
\]
 has norm one.  %

Consider instead the contraction 
\[
 G = \frac{F(1,1)+F(-1,1)}{2} 
\]
 and note that 
\[
 G = F(0,1) + \begin{pmatrix} 0&0\\0& |\ell_{1,1}|^2 \end{pmatrix} \otimes
  C_1^*C_1 + \begin{pmatrix} |\ell_{2,1}|^2 &0\\0&0\end{pmatrix}
   \otimes C_2 C_2^*.
\]
 Since $F(0,1)$ is positive semidefinite and has norm one,
 and $G$ has norm at most one, there is a (non-zero)
 vector $\dd =\sum_{j=1}^2 e_j \otimes \dd_j$ such that
 $F(0,1) \dd = \dd$ and both $\ell_{1,1} C_1 \dd_2=0$
 and $\ell_{2,1} C_2^* \dd_1=0.$ In particular, 
 $G\dd=\dd.$ Before continuing, note that not  both $\ell_{1,1}=0$
 and $\ell_{2,1}=0.$  

 Recall, since $\altL$ is invertible, not both of $\ell_{1,1}$
 and $\ell_{2,1}$ is zero. To prove at least one is not zero,
 observe that, since both $F(1,1)$ and $F(-1,1)$ have norm one
 and $G\dd=\dd,$ that $F(\pm 1,1)\dd=\dd.$ Hence
\[
\begin{split}
 0 = & (F(1,1)-F(-1,1))\dd \\
 & = \begin{pmatrix}  \left [|\ell_{2,1}|^2 +\ell_{2,1}\ell_{2,2}^*
  +\ell_{2,1}^* \ell_{2,2} \right ] C_2C_2^*\dd_1 
 + \left[ b_1^* \ell_{1,1} C_1^*C_1 + b_2^* \ell_{2,1} C_2C_2^*
 \right ] \dd_2\\
\left [ b_1 \ell_{1,1}^* C_1^*C_1 + b_2\ell_{2,1}^* C_2C_2^* \right ] \dd_1 
   +  \left [ |\ell_{1,1}|^2 + \ell_{1,1}\ell_{1,2}^* 
  +\ell_{1,1}^* \ell_{1,2} \right ] C_1^*C_1\dd_2 \end{pmatrix}.    
\end{split}
\]
 Thus,
\[
 0 = \begin{pmatrix} b_2^* \ell_{2,1} C_2C_2^* \dd_2 \\
    b_1^* \ell_{1,1} C_1^*C_1 \dd_1 \end{pmatrix}.
\]
 If neither $\ell_{1,1}=0$ nor $\ell_{2,1}=0,$ then
 $b_2^*C_2^*\dd_2=0,$ $b_1^* C_1 \dd_1=0$ as
 well as $C_1\dd_2=0$ and $C_2^*\dd_1,$
 in which case $F(0,1)\dd =0,$ a contradiction.

 Now suppose $\ell_{1,2}=0$ and $\ell_{2,2}\ne 0.$
 In this case $\ell_{1,1}\ne 0$ and equation~\eqref{e:elljj}
 with $j=1$ holds.  Since $\ell_{1,1}\ne 0,$ it follows
 that $\ell_{2,1}=0$ and  $C_1\dd_2=0=b_1^*C_1\dd_1.$ Thus
 $\varphi_1(\tX(0,1))\otimes C_1 \dd =0$ 
 and consequently, 
\[
1 =\| \varphi_2(\tX(0,1))\| = \| \begin{pmatrix} b_2 & \ell_{2,2}\\ 
  0 & b_2 \end{pmatrix} \|.
\]

Finally, suppose $\ell_{1,2}\ne 0$ and $\ell_{2,2}=0.$
 In this case equation~\eqref{e:elljk} holds with
 $(j,k)=(2,1).$ In particular, $\ell_{2,1}\ne 0$
 and $\ell_{1,1}=0.$  It follows that 
 $\varphi_2(\tX(0,1)) \otimes C_2^* \delta = 0$ and hence
\[
 1=\| \varphi_1(\tX(0,1))\| = \| \begin{pmatrix} b_1 & \ell_{1,2}\\
   0 & b_1 \end{pmatrix}\|.
\]
The proof is complete.
\end{proof}

\subsection{Caratheodory interpolation}
\label{ssec:Cint}
 This section presents a version of the Caratheodory interpolation
 theorem, Lemma~\ref{l:Cinterpolate},  and a consequence,
 Lemma~\ref{l:constraints}, that are used  in Subsection~\ref{s:order3}
  below to analyze the 
 action of $\varphi$ on tuples nilpotent of order three.

\begin{lemma}
\label{l:Cinterpolate}
  Suppose $b,\ell\in \C$ and $|b|<1.$  If 
\[
 \begin{pmatrix} b & \ell \\ 0 & b \end{pmatrix}
\]
has norm one, then there is a $\theta$ such that  
  $\ell = e^{i\theta}(|b|^2-1).$

Let $e=e^{i\theta}(|b|^2-1).$ If $f\in \CC$ and 
\[
 T=  \begin{pmatrix} b & e  & f \\ 0&b& e \\ 0&0&b\end{pmatrix}
\]
has norm at most one, then $T$ has norm one and $f=e^{i\theta} b^* e.$
Further, %
$TT^*=I-aa^*$ where
\[
 a= \sqrt{1-|b|^2} \begin{pmatrix} c^2\\ c \\ 1 \end{pmatrix} 
\]
 and $c=e^{i\theta} b^*.$
Likewise
 $T^*T=I-aa^*$ where
\[
 a= \sqrt{1-|b|^2}\, 
  \begin{pmatrix}  1\\   c \\ c^2 \end{pmatrix},
\]
 and where $c=\emit b.$
 In both cases, $\|a\| = \sqrt{1-|b|^{6}}<1.$
\end{lemma}

\begin{proof}
 Apply the Caratheodory interpolation theorem or
 simply verify by direct computation.
\end{proof}

 Lemmas~\ref{l:Cinterpolate} and \ref{l:eitheror} together
 reduce the proof of Theorem~\ref{t:mainR} to 
 two cases. Namely, using $\altL$ from equation~\eqref{d:L}, there 
 exists $\theta_1,\theta_2\in\mathbb{R}$ such that either
\begin{equation}
\label{case1}
\altL = \begin{pmatrix}  
  e^{i \theta_1}(|b_1|^2-1)&0\\0& e^{i\theta_2} (|b_2|^2-1) \end{pmatrix}
\end{equation}
 or 
\begin{equation}
\label{case2}
  \altL= \begin{pmatrix} 0 & e^{i\theta_1} (|b_1|^2-1) \\ e^{i\theta_2}(|b_2|-1) & 0
  \end{pmatrix}.
\end{equation}

\begin{lemma}
 \label{l:constraints}
 Suppose $\alpha,\beta, b\in\C$ and  $|b|<1.$ Let
 $e=\eit (|b|^2-1).$ %

If 
\[
 R=\begin{pmatrix} b & \alpha & \beta \\ 0 &b &e \\0&0&b\end{pmatrix}
\]
 is a contraction,  then $\|R\|=1$ and $(\alpha,\beta)$ is
 a multiple of $(1, \eit b^*).$  Equivalently,
 $\beta=\alpha \eit b^*.$

 Similarly, if
\[
 R=\begin{pmatrix} b& e & \beta \\ 0 & b & \alpha \\ 0 &0& b\end{pmatrix}
\]
 is a contraction, then $\|R\|=1$ and $\beta = \alpha \eit  b^*.$ 
\end{lemma}

\begin{proof}
Let 
\[
 C=\begin{pmatrix} b &e \\0 & b\end{pmatrix},
\]
and $v^* =\begin{pmatrix} \alpha & \beta \end{pmatrix}$ 
so that 
\[
 R=\begin{pmatrix} b & v^* \\ 0 & C \end{pmatrix}.
\]
Note that $I-C^*C = (1-|b|^2)ff^*$ where 
\[
 f=  \begin{pmatrix} 1 \\ \emit b \end{pmatrix}. 
\]
Since $R$ is a contraction, $I-C^*C - vv^* \succeq 0.$
 Hence $v$ is a multiple of $f.$  The proof of 
 the second part is similar and omitted.
\end{proof}

\begin{lemma}
\label{l:boundbelow}
 Suppose $b\in\C$ with $|b|<1,$  $\lambda,\theta\in \mathbb{R}$
 and $0<\lambda <1.$ Let $e=\eit (|b|^2-1)$ and
 $f=\eit b^* e$ and
\[
  X=\begin{pmatrix} b& \lambda e & \lambda  f \\ 0 & b & e \\
     0&0&b \end{pmatrix},  \ \ \ 
   Y = \begin{pmatrix} b& e & \lambda f \\ 0 & b & \lambda e \\
  0 &0 & b\end{pmatrix}.
\]
 For $P$ equal to any of $X^*X,$ $XX^*,$ $Y^*Y$ and $YY^*,$ 
 there is a two dimensional subspace
 $\cN\subseteq \C^3$ such that the compression of $P$ to $\cN$ is 
 bounded below by $\lambda^2.$ Further, if $b\ne 0,$
 then $P$ is strictly bounded below by $\lambda^2.$
\end{lemma}

\begin{proof}
Let
\[
 A = \begin{pmatrix} b & e & f \\ 
    0 & b & e \\ 0 & 0 & b \end{pmatrix}.
\]
By Lemma~\ref{l:Cinterpolate}  %
\[
 A^*A =I - aa^*,
\]
 for some $a\in\C^3$ of norm at most one. Setting
\[
 \Lambda =\begin{pmatrix} \lambda &0&0\\0&1&0\\0&0&1\end{pmatrix},
\]
observe  that $X=\Lambda A\Lambda^{-1}.$  Thus
\begin{equation}
\label{e:keyle}
\begin{split}
 X^*X = & \Lambda^{-1} A^* \Lambda^2 A \Lambda^{-1}
 \succeq \lambda^2  \Lambda^{-1} A^*A \Lambda^{-1}\\
& = \lambda^2 \Lambda^{-1} (I-aa^*) \Lambda^{-1} \succeq  
  \lambda^2 \left (I - cc^*\right ),
\end{split}
\end{equation}
where $c=\Lambda^{-1}a.$ 
Thus $X^*X -\lambda^2$ is positive 
 semidefinite on the orthogonal complement of $c.$

  If $\gamma$ is a unit vector that is  orthogonal to $c$ and 
 $\|X\gamma\|=\lambda,$ then the inequalities in
 equation \eqref{e:keyle} are equalities. Using $|\lambda|<1,$ from the second
 of these inequalities,  $\lambda^2 \| \Lambda^{-1} \gamma\|= \lambda^2.$
 Thus, writing $\gamma =\sum \gamma_j e_j$ with respect to
 the standard orthonormal basis, $\gamma_1=0$. In particular,
 $\Lambda^{-1}\gamma=\gamma.$ Equality in the first inequality
 now  implies $A\gamma = \kappa e_1$ for some unimodular $\kappa.$
 Hence $b=0$. Thus, if $b \ne 0$ and $h\in\C^3$ is orthogonal
 to $c,$ then $\|X h\|>\lambda \|h\|.$   
   The remaining statements can be proved 
 similarly. The details are omitted. 
\end{proof}

\subsection{Order three nilpotency}
\label{ssec:order3}
\label{s:order3}
Let \index{$\Theta$}
\begin{equation*}
 \Theta =\{(\lambda,\mu)  \in \C : |\lambda|^2 C_1^*C_1 + |\mu|^2 C_2 C_2^*
    \prec I\}.
\end{equation*}
For $\kappa,\lambda,\mu,\nu\in \CC$, let \index{$T=T(\kappa,\lambda,\mu,\nu)$}
\begin{equation}
\label{def:T+3}
T(\kappa,\lambda,\mu,\nu)=
 \left ( \begin{pmatrix} 0&\lambda&0\\0&0&\kappa\\0&0&0\end{pmatrix},
        \begin{pmatrix} 0&\nu&0\\0&0&\mu\\0&0&0\end{pmatrix} \right ).
\end{equation}

\begin{lemma}
\label{l:Theta}
  The set $\Theta$ is open and  contains a neighborhood of $0.$ 

  For  each $(\lambda,\mu) \in\Theta$ and $|\kappa|,|\nu|< 1$  the tuple 
  $T(\kappa,\lambda,\mu,\nu)$ 
 is in $\fP_{\EE}.$ 

 If $|\kappa|,|\nu|\le 1$
  and $(\lambda,\nu)$ is in the closure of
 $\Theta,$ and if either $|\kappa|=1,$ or  $|\nu|=1$
 or $(\lambda,\nu)$ is in the boundary of $\Theta,$
 then $T(\kappa,\lambda,\mu,\nu)$  is in the boundary of
 $\fP_{\EE}.$
\end{lemma}

\begin{proof}
Observe $I -  \left (T_1^*T_1 \otimes C_1^* C_1 + T_2T_2^* \otimes C_2C_2^*  \right )$
 is unitarily equivalent to 
\[
\begin{pmatrix} I-  |\nu|^2 C_2C_2^* &0&0\\0& I- \left (|\lambda|^2 C_1^*C_1 
   + |\mu|^2 C_2 C_2^*\right )
   \\ 0&0& I-|\kappa|^2 C_1^* C_1 \end{pmatrix}.
\]
\end{proof}

We now take up the two cases \eqref{i:a} and \eqref{i:b} of
 Lemma~\ref{l:eitheror}. 

\subsubsection{The $\ell_{1,2}=0=\ell_{2,1}$ case}
\label{sec:case12=0}
 Suppose, for the duration of this subsection, $\ell_{1,2}=0=\ell_{2,1}$
 and thus, as in equation~\eqref{case1},  
 $e_j:=\ell_{j,j} = e^{i\theta_j} (|b_j|^2-1).$   
 Let 
\begin{equation*}
  f_j= e^{2i \theta_j} b_j^* (|b_j|^2-1) =e^{i\theta_j}  b_j^* e_j.
\end{equation*}
\index{$e_j$} \index{$f_j$}

\begin{lemma}
\label{l:b1e1f1}
  If  $N=(N_1,N_2)$ is nilpotent of order three, then
\[
 \varphi_1(N)  = b_1 I + e_1 N_1 + f_1 N_1^2.
\]
\end{lemma}

\begin{proof}
  There exists $a_{j,k}$ for $1\le j,k\le 2$ such that, if
 $N=(N_1,N_2)$ is nilpotent of order three, then
\[
 \varphi_1(N) = b_1 + e_1 N_1 +\sum a_{j,k} N_j N_k.
\]

 Given $(\lambda,\mu)\in\Theta$ 
 or $(\lambda,\mu)\in \{(1,0),(0,1)\}$ and $|\nu|\le 1,$   the tuple
 $T=T(1,\lambda,\mu,\nu)$ of equation~\eqref{def:T+3} 
 lies in the boundary of $\fP_{\EE}$ by Lemma~\ref{l:Theta}. Hence
\[
 \varphi_1(T) = b_1 I + e_1 T_1 + \sum_{j,k=1}^2 a_{j,k} T_j T_k
\]
 has norm at most one.

 For $|\nu|\le 1$ and $T=T(1,1,0,\nu),$ %
\[
 \varphi_1(T) = \begin{pmatrix} b_1 & e_1 & a_{1,1} + \nu a_{2,1} \\
  0&b_1&e_1\\0&0&b_1 \end{pmatrix}.
\]
From Lemma~\ref{l:Cinterpolate} it follows that $a_{1,1} +\nu a_{2,1}=f_1$
 independent of $\nu.$
 Thus $a_{1,1}=f_1$ and $a_{2,1}=0.$

Choosing $T=(1,0,1,1),$
\[
 \varphi_1(T) = \begin{pmatrix} b_1 & 0 &  a_{2,2} \\
  0&b_1 & e_1 \\0&0&b_1 \end{pmatrix}
\]
is a contraction.
 Lemma~\ref{l:constraints}
 now implies $a_{2,2}=0.$

 Finally, given $(\lambda,\mu)\in \Theta,$
 and $T=T(1,\lambda,\mu,1)$, 
\[
 \varphi_1(T) =\begin{pmatrix} b_1 & \lambda e_1 & \lambda f_1 + 
   \lambda\mu a_{1,2} \\ 0 & b_1 & e_1 \\0&0&b_1 \end{pmatrix}.
\]
 By Lemma \ref{l:constraints}, %
\[
 e_1 \eit b_1^* = f_1 +\mu a_{1,2} 
\]
independent of $\mu.$ Thus  $a_{1,2}=0$
 and the proof is complete.
\end{proof}

\begin{lemma}
 \label{l:b2e2f2} If $N=(N_1, N_2)$ is nilpotent of order three,
  then \[ \varphi_2(N) = b_2 + e_2N_2 + f_2N_2^2. \]
\end{lemma}

\begin{proof}
 There exists $a_{j,k}\in\CC$ such that for $N=(N_1,N_2)$
 nilpotent of order three,
\[
 \varphi_2(T) = b_2 I + e_2 N_2 + \sum_{j,k=1}^2 a_{j,k} N_j N_k.
\]

 For $N=T=T(\kappa,0,1,1)$ and $|\kappa|\le 1,$ 
 \[
   \varphi_2(T) = \begin{pmatrix}
     b_2 & e_2 & \kappa a_{2,1} + a_{2,2} \\
     0 & b_2 & e_2 \\
     0 & 0 & b_2
   \end{pmatrix}
 \]
 is a contraction by Lemma~\ref{l:Theta}. 
 Lemma \ref{l:Cinterpolate} implies $\kappa a_{2,1} + a_{2,2} = e^{2i\theta_2}b^*(|b|^2-1) = f_2.$
 Hence $a_{2,1}=0$ and $a_{2,2}=f_2.$

 For $T = T(1,1,0,1)$ 
 \[
   \varphi_2(T) = \begin{pmatrix}
     b_2 & e_2 & a_{1,1} \\
     0 & b_2 & 0 \\
     0 & 0 & b_2
   \end{pmatrix}.
 \]
 From Lemma \ref{l:constraints}, it follows that $a_{1,1} = 0.$ 

 Finally, given $(\lambda, \mu) \in \Theta$ and
 choosing  $T = T(1,\lambda, \mu, 1),$
 \[
   \varphi_2(T) = \begin{pmatrix}
     b_2 & e_2 & \lambda\mu a_{1,2}+\mu f_2 \\
     0 & b_2 & \mu e_2 \\
     0 & 0 & b_2
   \end{pmatrix}.
 \]
 By Lemmas \ref{l:constraints} and \ref{l:Theta},
 \[
   e_2\eit b_2^* = f_2 + \lambda a_{1,2}
 \]
 independent of $\lambda.$ Hence, $a_{1,2} = 0$ and the proof is complete.
\end{proof}

The completion of the proof of Theorem~\ref{t:mainR}, 
in the present case, is at hand.
In view of Assumption \ref{a:notleI}, 
there is a $0<t<1$ such that 
\[
 t^2( C_1^* C_1 + C_2 C_2^* )\preceq I,
\]
 but, for $s>t,$
\[
s^2( C_1^* C_1 + C_2 C_2^* )\not\preceq I,
\]
 Hence, there is a unit vector $\gamma$ such that
\[
  t^2( C_1^* C_1 + C_2C_2^* )\gamma =\gamma.
\]
 Moreover, since $C_1$ and $C_2$ are contractions and $0<t<1,$
 we have $C_1\gamma\ne 0$ and $C_2^*\gamma\ne 0.$

Choose $T=T(1,t,t,1)$ and note that $T$ is nilpotent and in the boundary 
 of $\fP_{\EE}$ by Lemma~\ref{l:Theta}.  Thus, by Lemma~\ref{l:deltodel}, 
$R=\varphi(T)$ is in the boundary of $\fP_{\EE}.$  Let $R_j=\varphi_j(T).$ Thus,
 using Lemma~\ref{l:b1e1f1},
\[
 R_1 = \begin{pmatrix} b_1 & t e_1 & tf_1\\ 0 & b_1 & e_1\\ 0 &0&b_1\end{pmatrix}.
\]
 By Lemma~\ref{l:boundbelow}, there is a two dimensional
 subspace $\cN$ of $\C^3$ on which $R_1^*R_1$  is bounded below
 by $t^2,$ and strictly so if $b_1\ne 0.$ Likewise, 
 using Lemma~\ref{l:deltodel} and Lemma~\ref{l:b2e2f2},
\[
R_2 = \begin{pmatrix} b_2 & e_2 & tf_2\\ 0&b_2 &t e_2\\0&0&b_2\end{pmatrix}.
\]
By Lemma~\ref{l:boundbelow}, there is a two dimensional
 subspace $\cM$ of $\C^3$ on which $R_2R_2^*$ is bounded below
 by $t^2,$ and strictly so if $b_2\ne 0.$ Hence there is a
 unit vector $\Gamma\in \cN\cap \cM \subset \C^3$
 such that
\[
 \langle R_1\Gamma \,R_1\Gamma \rangle, \langle R_2^* \Gamma,R_2^*\Gamma \rangle 
 \ge  t^2 \|\Gamma \|^2 =t^2,
\]
 with strict inequality holding in the case $b_j\ne 0$.
It follows that, if either $b_1\ne 0$ or $b_2\ne 0,$ then
\[
\langle [R_1^*R_1\otimes C_1^* C_1 + R_2R_2^* \otimes C_2C_2^* ]
  \Gamma\otimes \gamma,\Gamma\otimes\gamma \rangle =
   \|R_1\Gamma\|^2 \, \|C_1\gamma\|^2 + \|R_2^*\Gamma\|^2 \, \|C_2^*\gamma\|^2 >1
\]
and thus $R=(R_1,R_2)$ is not in the closure of  $\fP_{\EE}$ 
 by Lemma~\ref{l:inDA-alt},
contradicting Lemma~\ref{l:deltodel} that says 
 $\varphi$ maps nilpotent elements of
 the boundary of $\fP_{\EE}$ into the  boundary (closure) of $\fP_{\EE}.$
This contradiction shows $b=0.$

 Since $b=0$ and the  domain $\fP_{\EE}$ is circular,  $\varphi$ is linear
 by  \cite[Theorem~4.4]{proper}.  Thus,
\[
 \varphi(x) 
   = - \begin{pmatrix} \eito x_1,  &  e^{i\theta_2} x_2 \end{pmatrix}.
\]

 \subsubsection{The case $\ell_{j,j}=0$}
\label{sec:casejj=0}
 In the case  \eqref{i:b} in Lemma~\ref{l:eitheror},  $\ell_{j,j}=0$ for $j=1,2$ and,
 from equation~\eqref{case2},
\[
 \ell_{1,2} =   e^{i\theta_1} (|b_1|^2-1), \ \ 
  \ell_{2,1} =  e^{i\theta_2}(|b_2|-1).
\]
 We will argue that these assumptions, together with
 Assumption~\ref{a:notleI}, leads to a contradiction.

 For notational convenience, let $\eo=\ell_{1,2}=\eito (|b_1|^2-1)$  and 
 $\fo=e^{2i\theta_1} b_1^*(|b_1|^2-1)= \eito b_1^* \eo.$
 In this case, for tuples $N$ nilpotent of order three
\[
 \varphi_1(N)  =  \varphi_1(N) =  b_1 + \eo N_2 + \sum_{j,k=1}^2 a_{j,k} N_jN_k,
\]
 for some $a_{j,k}\in\C.$

Since $T=T(0,0,1,1)$ (from equation~\eqref{def:T+3})
 is in the boundary of $\fP_{\EE},$ 
\[
 \varphi_1(T)= \begin{pmatrix} b_1 & e_1 & a_{2,2}
   \\ 0  & b_1 & e_1  \\ 0&0&b_1 \end{pmatrix}
\]
 is a contraction.
 It follows from Lemma \ref{l:Cinterpolate} 
  that $a_{2,2} =e^{2i\theta_1}b_1^*(|b_1|^2-1)= \fo.$

Next consider $T=T(1,0,0,1).$
 This tuple lies in the boundary of $\fP_{\EE}$ and hence,
\[
 \varphi_1(T) = \begin{pmatrix} b_1 & e_1  & a_{2,1} 
   \\ 0  & b_1 & 0 \\ 0&0&b_1 \end{pmatrix}
\]
 has norm at most one. A variation on Lemma \ref{l:constraints}
 implies  $a_{2,1}=0.$

Now consider $T=(1,1,0,1).$ 
 Again $T$ is in the boundary of $\fP_{\EE}$ and hence (using $a_{2,1}=0$)
\[
 \varphi_1(T) =\begin{pmatrix} b_1 & e_1  & a_{1,1}\\
 0&b_1 &0\\0&0&b_1 \end{pmatrix}
\]
 is a contraction.
  Hence  $a_{1,1}=0.$

 Showing  $a_{1,2}=0$ involves a bit more computation. 
 For $(\lambda,\mu)\in\Theta$, the tuple $T=(0,\lambda,\mu,1)$
 lies in the boundary of $\fP_{\EE}.$ Hence,
\[
 \varphi_1(T) = \begin{pmatrix} b_1 & \eo & \mu (\lambda a_{1,2} +\fo)\\
  0&b_1& \mu \eo \\ 0&0& b_1 \end{pmatrix}
\]
 has norm at most one. 
 A variation of Lemma \ref{l:constraints} now implies 
 that $a_{1,2}=0.$ Thus,  for tuples $N$ nilpotent of order three,
\[
 \varphi_1(N) = b_1 + \eo N_2 + \fo  N_2^2.
\]

 A similar argument shows, with $\et=\ell_{2,1}$  and 
 $\ft=e^{2i\theta_2} b_2^*(|b_2|^2-1)= e^{i\theta_2}  b_2^* \et,$
\[
 \varphi_2(N) = b_2 + \et N_1 + \ft N_1^2,
\]
 for tuples $N=(N_1,N_2)$ nilpotent of order three. 
 Indeed, 
\[
 \varphi_2(N) = b_2 + e_2 N_1 + \sum a_{j,k}N_jN_k.
\]
 Choosing $N=T=T(1,1,0,\nu)$ with $|\nu|\le 1,$
\[
 \varphi_2(T) = \begin{pmatrix} b_2 & e_2 & a_{1,1}+ \nu a_{2,1} \\
                   0 & b_2 & e_2 \\ 0& 0& b_2 \end{pmatrix}
\]
 from which it follows that $a_{1,1}=f_2$ and $a_{2,1} =0.$
 Choosing $T=T(1,0,1,1)$ shows $a_{2,2}=0.$ Finally,
with $(\lambda,\mu)\in\Theta,$ choosing $T=T(1,\lambda,\mu,1)$
 shows $a_{1,2}=0.$

In view of Assumption \ref{a:notleI}, 
there is a $0<t<1$  and unit vector $\gamma$ such that 
 both 
\[
 t^2( C_1^* C_1 + C_2 C_2^* )\preceq I,
 \ \ \   t^2( C_1^* C_1 + C_2^* C_2 )\gamma =\gamma
\]
and  $C_1\gamma\ne 0$ and $C_2^*\gamma\ne 0.$

As before, the tuple $T=T(1,t,t,1)$ is in the boundary of 
 $\fP_{\EE}.$ Hence, $R=\varphi(T)$ is also.   Let $R_j=\varphi_j(T).$
Thus,
\[
 R_1 = b_1 + e_1 T_2+f_1 T_2^2 
  = \begin{pmatrix} b_1 & e_1 & t f_1 \\ 0 & b_1 & t e_1 \\
   0&0&b_1\end{pmatrix}.
\]
 By Lemma~\ref{l:boundbelow}, there is a two dimensional
 subspace $\cM\subset \C^3$ on which $R_1^*R_1$ is bounded below
 by $t^2$ and strictly so if $b_1\ne 0.$
Similarly,
\[
 R_2 = \begin{pmatrix} b_2 & te_2 & t f_2 \\
   0 & b_2 & e_2 \\0&0&b_2\end{pmatrix}
\]
and, by Lemma~\ref{l:boundbelow}, there is a two
dimensional subspace $\cN\subset \C^3$ on which
 $R_2R_2^*$ is bounded below by $t^2$ and
strictly so if $b_2\ne 0.$
Hence  there is a unit vector $\Gamma\in \C^3$
 such that
\[
 \langle R_1\Gamma \,R_1\Gamma \rangle, \langle R_2^* \Gamma,R_2^*\Gamma \rangle 
 \ge  t^2 \|\Gamma\|^2 =t^2,
\]
 with strict inequality holding in the case $b_j\ne 0$.
It follows that if either $b_1\ne 0$ or $b_2\ne 0,$ then
\[
\langle [R_1^*R_1\otimes C_1^* C_1 + R_2^*R_2 \otimes C_2C_2^* ]
  \Gamma\otimes \gamma,\Gamma\otimes\gamma \rangle =
   \|R_1\Gamma\|^2 \, \|C_1\gamma\|^2 + \|R_2^*\Gamma\|^2 \, \|C_2^*\gamma\|^2 >1.
\]
Thus $(R_1,R_2)$ is not in the closure of  $\fP_{\EE}$ by Lemma~\ref{l:inDA-alt},
 contradicting Lemma~\ref{l:deltodel} that says $\varphi$
 maps nilpotent elements of
 the boundary of $\fP_{\EE}$ into the  boundary (closure) of $\fP_{\EE}.$
 Hence $b_1=0=b_2;$ that is $b=0.$

 Since $\varphi(0)=0$  and $\fP_{\EE}$ is  circular, 
by \cite[Theorem~4.4]{proper}, $\varphi$ is linear. Hence,
 \[
 \varphi(x) %
   = - \begin{pmatrix} \eito x_2,  e^{i\theta_2} x_1 \end{pmatrix}.
\]

  By Lemma~\ref{l:freethebidisk}, 
\[
 T=(T_1,T_2) = \left (  \begin{pmatrix} 0&0&0\\0&0&1\\0&0&0\end{pmatrix},
   \begin{pmatrix} 0&1&0\\0&0&0\\0&0&0 \end{pmatrix} \right),
\]
 is in the closure of $\fP_{\EE}.$ 
 On the other hand, 
\[
 T_2^*T_2 \otimes C_1^*C_1 + T_1T_1^* \otimes C_2C_2^* 
  =\begin{pmatrix} 0&0&0 \\ 0&C_1^* C_1 + C_2 C_2^* & 0 \\ 0& 0 &0 \end{pmatrix}
 \not\preceq  I
\]
 and therefore $\varphi(T)=(e^{i\theta_2} T_2, e^{i\theta_1} T_1)$ is not in the closure of $\fP_{\EE},$
 and we have reached a contradiction (of Lemma~\ref{l:deltodel})
  that completes the proof.


\newpage

\printindex

\end{document}